\documentclass[twoside,a4paper,reqno,11pt]{amsart} 
\usepackage[top=28mm,right=28mm,bottom=28mm,left=28mm]{geometry}
\usepackage{microtype}
\usepackage{amsfonts, amsmath, amssymb, amsbsy, mathrsfs, bm, latexsym, stmaryrd, hyperref, array, enumitem, textcomp, url}
\usepackage[table]{xcolor}

\renewcommand{\a}{\alpha}
\renewcommand{\b}{\beta}

\newcommand{\e}{\epsilon}

\renewcommand{\O}{\Omega}

\newcommand{\la}{\langle}
\newcommand{\ra}{\rangle}
\newcommand{\leqs}{\leqslant}
\newcommand{\geqs}{\geqslant}

\newcommand{\vs}{\vspace{3mm}}

\makeatletter
\newcommand{\imod}[1]{\allowbreak\mkern4mu({\operator@font mod}\,\,#1)}
\makeatother

\theoremstyle{plain}
\newtheorem{theorem}{Theorem} 
 
\newtheorem{corol}[theorem]{Corollary}
\newtheorem{thm}{Theorem}[section] 
\newtheorem{lem}[thm]{Lemma}
\newtheorem{prop}[thm]{Proposition}

\newtheorem*{theorem*}{Theorem} 
\newtheorem*{conj*}{Conjecture}

\theoremstyle{definition}
\newtheorem{rem}[thm]{Remark}

\newtheorem{ex}[thm]{Example}
\newtheorem{remk}{Remark}

\begin{document}

\title[On the generation of simple groups by Sylow subgroups]{On the generation of simple groups \\ by Sylow subgroups}

\author{Timothy C. Burness}
\address{T.C. Burness, School of Mathematics, University of Bristol, Bristol BS8 1UG, UK}
\email{t.burness@bristol.ac.uk}

\author{Robert M. Guralnick}
\address{R.M. Guralnick, Department of Mathematics, University of Southern California, Los Angeles, CA 90089-2532, USA}
\email{guralnic@usc.edu}

\date{\today} 

\begin{abstract}
Let $G$ be a finite simple group of Lie type and let $P$ be a Sylow $2$-subgroup of $G$. In this paper, we prove that for any nontrivial element $x \in G$, there exists $g \in G$ such that $G = \la P, x^g \ra$. By combining this result with recent work of Breuer and Guralnick, we deduce that if $G$ is a finite nonabelian simple group and $r$ is any prime divisor of $|G|$, then $G$ is generated by a Sylow $2$-subgroup and a Sylow $r$-subgroup.
\end{abstract}

\maketitle

\section{Introduction}\label{s:intro}

Let $G$ be a finite group, let $\pi$ be a set of primes and let $\pi'$ be the complementary set of primes. In \cite{BG}, Breuer and Guralnick prove that $G$ can be generated by a $\pi$-subgroup and a $\pi'$-subgroup of $G$ (moreover, each generating subgroup $H$ can be chosen to be intravariant, which means that $H$ and $\varphi(H)$ are conjugate in $G$ for all automorphisms $\varphi$ of $G$). By taking $\pi = \{2\}$ it follows that $G$ can be generated by a Sylow $2$-subgroup and an intravariant (solvable) subgroup of odd order. 

A key ingredient in the proof of this theorem is a stronger result for finite simple groups (see \cite[Theorem 1.8]{BG}), which relies on the classification of finite simple groups. This asserts that if $G$ is simple then there exists a prime $r$ (depending on $G$) such that $G$ can be generated by a Sylow $r$-subgroup and a Sylow $s$-subgroup for any prime divisor $s$ of $|G|$. This builds on earlier work of Aschbacher and Guralnick \cite{AG}, who proved that every simple group $G$ can be generated by a pair of Sylow $p$-subgroups for some prime $p$. In turn, this was extended in \cite{Gur}, where Guralnick shows that $G$ is generated by a Sylow $2$-subgroup and an involution (in particular, $G$ is generated by a pair of Sylow $2$-subgroups). 

In fact, Breuer and Guralnick conjecture that if $G$ is simple and $r,s$ are \emph{any} prime divisors of $|G|$ (not necessarily distinct) then $G$ can be generated by a Sylow $r$-subgroup and a Sylow $s$-subgroup. A proof for all alternating and sporadic simple groups is given in \cite{BG}. In addition, \cite[Corollary 19]{BGG} implies that if $r$ and $s$ are fixed primes then the conclusion of the conjecture holds for all but finitely many finite simple groups. The full version of the conjecture remains open. 

Our main result is the following. 

\begin{theorem}\label{t:main}
Let $G$ be a finite simple group of Lie type over $\mathbb{F}_q$ and let $P$ be a Sylow $2$-subgroup of $G$. Then for all nontrivial $x \in G$, there exists $g \in G$ such that $G = \la P, x^g \ra$.
\end{theorem}

By combining this with results in \cite{BG} for alternating and sporadic groups, we obtain the following corollary, which establishes a special case of the Breuer-Guralnick conjecture. It also shows that the prime in \cite[Theorem 1.8]{BG}, which a priori depends on the choice of the simple group, can in fact be chosen uniformly as $2$.

\begin{corol}\label{c:main}
Let $G$ be a finite nonabelian simple group and let $r$ be a prime divisor of $|G|$. Then $G$ is generated by a Sylow $2$-subgroup and a Sylow $r$-subgroup of $G$.
\end{corol}

\begin{remk}\label{r:main}
Let us record some comments on the statement of Theorem \ref{t:main}.
\begin{itemize}\addtolength{\itemsep}{0.2\baselineskip}
\item[{\rm (a)}] By \cite[Proposition 2.2]{BG}, the conclusion to Theorem \ref{t:main} holds if $G$ is a sporadic simple group. In fact, in this setting Breuer and Guralnick have used a computational approach to verify that the same conclusion holds when $P$ is a Sylow $p$-subgroup of $G$ and $p$ is an arbitrary prime divisor of $|G|$ (see \cite[Section 8.19]{GAPCTL}).
\item[{\rm (b)}] Let $G$ be a finite simple group of Lie type in characteristic $p$. If $P$ is a Sylow $p$-subgroup and $x \in G$ is nontrivial, then \cite[Proposition 2.3]{BG} states that $G = \la P, x^g \ra$ for some $g \in G$. So in order to prove Theorem \ref{t:main} we may assume $q$ is odd. 
\item[{\rm (c)}] As noted in \cite{BG}, Theorem \ref{t:main} is false for alternating groups. For example, let $G = A_n$ and let ${\rm orb}(P)$ be the number of $P$-orbits on $[n]=\{1, \ldots, n\}$ of a Sylow $p$-subgroup $P$ of $G$. Clearly, if $x \in G$ and ${\rm orb}(x) > n-{\rm orb}(P)$, where ${\rm orb}(x)$ is the number of orbits of $\la x \ra$ on $[n]$, then $\la P,x^g \ra$ is intransitive for all $g \in G$. For example, if $G = A_{15}$, $P$ is a Sylow $2$-subgroup and $x$ is a $3$-cycle, then ${\rm orb}(x) = 13$ and ${\rm orb}(P)=4$, so $G \ne \la P, x^g\ra$ for all $g \in G$.
\item[{\rm (d)}] It is also worth noting that the natural extension of Theorem \ref{t:main} to odd primes is false in general. For example, let $P$ be a Sylow $r$-subgroup of $G = {\rm L}_n(2)$, where $n = 2k-1$ and $r=2^k-1$ is a prime. Then $P = \la y \ra$ has order $r$ and $\dim C_V(y) = k-1$, so $G \ne \la P, x \ra$ for any element $x$ with $\dim C_V(x) \geqs k+1$.
\end{itemize}
\end{remk}

Our proof of Theorem \ref{t:main} relies heavily on a probabilistic approach based on fixed point ratio estimates. Computational methods also play a key role and we will also use character theory and some elementary linear algebra to handle certain special cases that arise. In order to explain the general set-up, let $G$ be a finite group, let $P \ne G$ be a Sylow $2$-subgroup of $G$ and let $\mathcal{M}$ be the set of maximal subgroups of $G$ containing $P$. Fix a nontrivial element $x \in G$ and let $\mathcal{Q}(x)$ be the probability that $G \ne \la P, x^g \ra$, where $x^g$ is a randomly chosen conjugate of $x$. Then
$\mathcal{Q}(x) \leqs \Sigma(x)$, where
\begin{equation}\label{e:sigma}
\Sigma(x) = \sum_{H \in \mathcal{M}} {\rm fpr}(x,G/H)
\end{equation}
and
\[
{\rm fpr}(x,G/H) = \frac{|x^G \cap H|}{|x^G|}
\]
is the fixed point ratio of $x$ with respect to the action of $G$ on $G/H$. In particular, if $\Sigma(x)<1$ then $x^G \ne \bigcup_{H \in \mathcal{M}} (x^G \cap H)$ and so there exists a conjugate $y \in x^G$ that is not contained in any maximal overgroup of $P$. Therefore $G = \la P, y \ra$ and so it suffices to show that $\Sigma(x)<1$ for all nontrivial $x \in G$. 

Since ${\rm fpr}(x,G/H) \leqs {\rm fpr}(x^m,G/H)$ for all $m \in \mathbb{Z}$, we only need to verify the bound $\Sigma(x)<1$ for elements of prime order. Therefore, in view of Remark \ref{r:main}(b), the following result reduces the proof of Theorem \ref{t:main} to the groups in part (ii), which can be handled using a different approach (see Theorem \ref{t:psln3}).

\begin{theorem}\label{t:main2}
Let $G$ be a finite simple group of Lie type in odd characteristic and let $P$ be a Sylow $2$-subgroup of $G$. Then either 
\begin{itemize}\addtolength{\itemsep}{0.2\baselineskip}
\item[{\rm (i)}] $\Sigma(x) < 1$ for all $x \in G$ of prime order; or
\item[{\rm (ii)}] $G = {\rm L}_n(3)$ and $n \geqs 7$ is odd.
\end{itemize}
\end{theorem}

\begin{remk}
We expect that the bound in (i) also holds for the groups in (ii), but this would require  sharper fixed point ratio estimates for subspace actions and it is easier to verify Theorem \ref{t:main} directly in this special case. For instance, if $G = {\rm L}_n(3)$ and $n \geqs 7$ is a Mersenne prime then $\mathcal{M} = \{P_m \,:\, 1 \leqs m < n\}$, where $P_m$ is the stabilizer of an $m$-dimensional subspace of the natural module for $G$. By \cite[Proposition 3.1]{GK} we have ${\rm fpr}(x,G/P_m)<2\cdot 3^{-m}$ if $m<n/2$ and thus 
\[
\Sigma(x) < 4\sum_{m=1}^{(n-1)/2}3^{-m},
\]
which does not yield $\Sigma(x)<1$. If $n=7$ and $x = (-I_6,I_1)$ then with the aid of {\sc Magma} we compute $\Sigma(x) = 1086/1093$. 
\end{remk}

In order to prove Theorem \ref{t:main2}, we first need to determine the subgroups in $\mathcal{M}$ and then we need to apply upper bounds on fixed point ratios for primitive actions of simple groups of Lie type. Here $\mathcal{M}$ coincides with the set of odd-index maximal subgroups of $G$ and the various possibilities have been determined by Liebeck and Saxl \cite{LS85} (and independently by Kantor \cite{Kantor}). We will also appeal to more recent work of Maslova \cite{Mas}, which gives a precise description of the odd-index maximal subgroups of simple classical groups in odd characteristic. Given a certain type of subgroup $H$ in $\mathcal{M}$ (for example, the stabilizer of an orthogonal decomposition of the natural module if $G$ is classical), we need to know the number of conjugacy classes of such subgroups in $G$ and we also need to estimate the number of distinct conjugates of $H$ that contain our given Sylow $2$-subgroup $P$. In each case, the number of conjugacy classes is readily available in the literature (for example, this can be read off from \cite{BHR,KL} when $G$ is classical).  And if $H \in \mathcal{M}$ then it is straightforward to show that $P$ is contained in exactly $|N_G(P):N_H(P)|$ distinct conjugates of $H$ (see Lemma \ref{l:count}). In particular, $P$ is contained in at most $|N_G(P):P|$ conjugates of $H$ and this allows us to apply work of Kondrat'ev and Mazurov \cite{Kon, KM} on the normalizers of Sylow $2$-subgroups in simple groups. 

There is an extensive literature on fixed point ratios for groups of Lie type. This includes general bounds, such as the main results in \cite{BG_FPR,LS91}, as well as more specialized results in \cite{Bur1,Bur2,Bur3,Bur4,GK} for classical groups and \cite{LLS} for exceptional groups. It will also be convenient to use {\sc Magma} \cite{Magma} to handle certain low rank groups defined over small fields, which allows us to compute $\Sigma(x)$ precisely for all $x \in G$ of prime order.

\begin{remk}
It is worth noting that the proof of \cite[Theorem 1.8]{BG} for classical groups is more constructive than our proof of Theorem \ref{t:main2}. In this paper, with some additional work, one could bound the probability $\mathcal{Q}(x)$ away from $0$ for all nontrivial $x \in G$, but we do not pursue this here.
\end{remk}

\vs

\noindent \textbf{Acknowledgements.} Both authors thank the Department of Mathematics at the California Institute of Technology for their generous hospitality during a research visit in spring 2022. They also thank an anonymous referee for their careful reading of the paper and for several helpful comments and suggestions. Guralnick was partially supported by the NSF grant DMS-1901595 and a Simons Foundation Fellowship 609771. 

\section{Exceptional groups}\label{s:excep}

In this section we prove Theorem \ref{t:main2} for exceptional groups of Lie type, while the classical groups will be handled in Section \ref{s:class}. All logarithms in this paper are in base $2$ and we adopt the standard notation for simple groups used in \cite{KL}.

Let $G$ be a finite simple group of exceptional Lie type over $\mathbb{F}_q$, where $q$ is odd. Let $P$ be a Sylow $2$-subgroup of $G$. Let $\mathcal{M}$ be the set of maximal subgroups of $G$ containing $P$ and note that each $H \in \mathcal{M}$ has odd index in $G$, so the possibilities for $H$ are described by Liebeck and Saxl in \cite{LS85}. We will also need the following theorem of Kondrat'ev and Mazurov (see \cite[Theorem 6]{KM}). Given a positive integer $n$, let $n_{2'}$ be the largest odd divisor of $n$.

\begin{thm}\label{t:km}
Let $G$ be a finite simple group of exceptional Lie type over $\mathbb{F}_q$ with $q$ odd and let $P$ be a Sylow $2$-subgroup of $G$. Then either $N_G(P) = P$, or one of the following holds:
\begin{itemize}\addtolength{\itemsep}{0.2\baselineskip}
\item[{\rm (i)}] $G = E_6^{\e}(q)$ and $N_G(P) = P \times C_m$, where $m = (q-\e)_{2'}/(3,q-\e)$. 
\item[{\rm (ii)}] $G = {}^2G_2(q)'$ and $N_G(P) = P{:}L$, where $L = C_7$ if $q=3$ and $L = C_7{:}C_3$ if $q > 3$.
\end{itemize}
\end{thm}

Given a subgroup $H \in \mathcal{M}$, let $n(H,P)$ be the number of distinct conjugates of $H$ containing $P$. The following elementary observation will be useful.

\begin{lem}\label{l:count}
Let $G$ be a finite group, let $P$ be a Sylow $p$-subgroup of $G$ and let $H$ be a subgroup of $G$ containing $P$. Then $n(H,P) = |N_G(P):N_H(P)|$.
\end{lem}

\begin{proof}
Suppose $P \leqs H^g$ for some $g \in G$. Then $P$ and $P^{g^{-1}}$ are Sylow $p$-subgroups of $H$, so Sylow's theorem implies that $P^{g^{-1}} = P^h$ for some $h \in H$ and thus $g \in Hx$ for some $x \in N_G(P)$. We deduce that $n(H,P)$ is the number of distinct cosets of the form $Hx$ with $x \in N_G(P)$, whence
\[
n(H,P) = |HN_G(P) : H| = |N_G(P):N_H(P)|
\]
as required. 
\end{proof}

\begin{thm}\label{t:ex}
The conclusion to Theorem \ref{t:main2} holds if $G$ is an exceptional group. 
\end{thm}

\begin{proof}
Let $G$ be a finite simple exceptional group of Lie type over $\mathbb{F}_q$, where $q=p^f$ and $p$ is an odd prime. Fix a Sylow $2$-subgroup $P$ of $G$ and define 
$\mathcal{M}$ and $n(H,P)$ as above, noting that $n(H,P) \leqs |N_G(P):P|$ by Lemma \ref{l:count}. Let $x \in G$ be an element of prime order and define $\Sigma(x)$ as in \eqref{e:sigma}. Our aim is to verify the bound $\Sigma(x)<1$.

First assume $G = E_8(q)$ and let $H \in \mathcal{M}$. By inspecting the main theorem of \cite{LS85}, we deduce that one of the following holds (up to conjugacy in $G$):
\begin{itemize}\addtolength{\itemsep}{0.2\baselineskip}
\item[{\rm (a)}] $H = E_8(q_0)$ is a subfield subgroup, where $q = q_0^e$ and $e$ is an odd prime.
\item[{\rm (b)}] $H$ is a maximal rank subgroup of type $D_8(q)$, $D_4(q)^2$, $A_1(q)^8$ or $(C_{q-\e})^8$ with $\e = \pm$.
\end{itemize} 
By Theorem \ref{t:km}, we have $n(H,P) = 1$ and \cite[Theorem 1]{LLS} gives ${\rm fpr}(x,G/H) \leqs q^{-8}(q^4-1)^{-1}$. Since there are at most $\log\log q$ odd prime divisors of $f$, it follows that  
\[
\Sigma(x) < (5+\log\log q)q^{-8}(q^4-1)^{-1} < 1 
\]
and the result follows.
 
Next assume $G = E_6^{\e}(q)$ and set $\widetilde{G} = {\rm Inndiag}(G)$ and $d = (3,q-\e)$, so $\widetilde{G} = G.d$ is the subgroup of ${\rm Aut}(G)$ generated by the inner and diagonal automorphisms. By \cite{LS85}, each $H \in \mathcal{M}$ is either a subfield subgroup of type $E_6^{\e}(q_0)$, where $q=q_0^e$ and $e \geqs 3$ is an odd prime divisor of $f$ (for each $e$ there is a unique $\widetilde{G}$-class of such subgroups, so at most $d$ distinct $G$-classes) or $H$ is $\widetilde{G}$-conjugate to one of up to $4$ maximal rank subgroups. Now Theorem \ref{t:km} gives $n(H,P) \leqs (q-\e)/2d$ and \cite[Theorem 1]{LLS} implies that ${\rm fpr}(x,G/H) \leqs (q^4-q^2+1)^{-1}$ for all nontrivial $x \in G$. Putting this together, we deduce that 
\[
\Sigma(x) < \frac{1}{2}(q+1)(4+\log\log q)(q^4-q^2+1)^{-1} < 1
\]
for all $q \geqs 3$.

The remaining cases are entirely similar and we omit the details since no special difficulties arise. Note that if $G = {}^2G_2(3)'$ then $\mathcal{M} = \{H\}$, where $H = N_G(R) = 2^3{:}7$, and we compute ${\rm fpr}(x,G/H) \leqs 2/9$ for all nontrivial $x \in G$.
\end{proof}

\section{Classical groups}\label{s:class}

Here we complete the proof of Theorem \ref{t:main} by handling the classical groups. We adopt the same notation as in the previous section and we begin by recalling the following theorem of Kondrat'ev \cite{Kon}. Given a positive integer $n$, let $t(n)$ be the number of nonzero digits in the binary expansion of $n$. In part (ii), note that ${\rm PSp}_2(q) = {\rm L}_2(q)$.

\begin{thm}\label{t:kon}
Let $G$ be a finite simple classical group over $\mathbb{F}_q$ with $q$ odd and let $P$ be a Sylow $2$-subgroup of $G$. Then either $N_G(P) = P$, or one of the following holds:
\begin{itemize}\addtolength{\itemsep}{0.2\baselineskip}
\item[{\rm (i)}] $G = {\rm L}_n^{\e}(q)$, $n \geqs 3$, $t(n) \geqs 2$ and $|N_G(P):P| = ((q-\e)_{2'})^{t(n)-1}/(n,q-\e)_{2'}$.
\item[{\rm (ii)}] $G = {\rm PSp}_{n}(q)$, $n \geqs 2$, $q \equiv \pm 3 \imod{8}$ and $|N_G(P):P| = 3^{t(n)}$.
\end{itemize}
\end{thm}

\begin{rem}\label{r:matrix}
For certain low rank classical groups defined over small fields, we can compute $\Sigma(x)$ precisely with the aid of {\sc Magma} \cite{Magma} (version V2.26-6). To do this, we typically work in the quasisimple version of the group, namely $L = {\rm SL}_n^{\e}(q)$, ${\rm Sp}_n(q)$ or $\O_n^{\e}(q)$, and we use the functions \texttt{ClassicalClasses} and \texttt{ClassicalMaximals} to construct representatives of the conjugacy classes of elements and maximal subgroups of $L$. It is then straightforward to calculate the relevant fixed point ratios and we compute $\max\{\Sigma(x) \,:\, 1 \ne x \in G\}$ by running over a set of representatives of the conjugacy classes of non-central elements in $L$. 
\end{rem}

\begin{ex}\label{e:magma}
Suppose $G = {\rm P\O}_8^{-}(5)$. We can use the following {\sc Magma} code to show that $\Sigma(x) \leqs 17629/203763$ for all $1 \ne x \in G$. Here we are using the fact that $n(H,P) = 1$ for all $H \in \mathcal{M}$ (see Theorem \ref{t:kon}). We also use  the function \texttt{FixedPointRatio} defined in \cite[Section 1.2.1]{BH}.

{\small
\begin{verbatim}
G:=OmegaMinus(8,5);
cl:=ClassicalClasses(G);
M:=ClassicalMaximals("O-",8,5); 
S:=SylowSubgroup(G,2); 
A:=[i : i in [1..#M] | #M[i] mod #S eq 0];
B:=[];
a:=[i : i in [1..#cl] | cl[i][2] gt 1];
for i in a do
  z:=0; 
  x:=cl[i][3];
  for j in A do 
    z:=z+FixedPointRatio(G,M[j],x); 
  end for;
  Append(~B,z); 
end for;
Maximum(B);
\end{verbatim}}
\end{ex}

We divide the proof of Theorem \ref{t:main2} for classical groups into several subsections and we begin by handling the orthogonal groups in Sections \ref{s:o_odd} and \ref{s:o_even}.

\subsection{Odd dimensional orthogonal groups}\label{s:o_odd}

Let $G = \O_n(q)$, where $n \geqs 7$ and $q \geqs 3$ are odd. Let $P$ be a Sylow $2$-subgroup of $G$ and let $H \in \mathcal{M}$ be a maximal overgroup of $P$ in $G$. By Theorem \ref{t:kon} we have $n(H,P) = 1$ and by appealing to \cite{LS85,Mas} we deduce that one of the following holds:
\begin{itemize}\addtolength{\itemsep}{0.2\baselineskip}
\item[{\rm (a)}] $H$ is a subfield subgroup of type ${\rm O}_{n}(q_0)$, where $q=q_0^e$ for some odd prime $e$.
\item[{\rm (b)}] $H$ is the stabilizer of an orthogonal decomposition $V = V_1 \perp \cdots \perp V_n$, where the $V_i$ are isometric $1$-spaces and $q=p \equiv \pm 3 \imod{8}$.
\item[{\rm (c)}] $n=7$, $q=p \equiv \pm 3 \imod{8}$ and $H = {\rm Sp}_6(2)$.
\item[{\rm (d)}] $H = G_U$ is the stabilizer in $G$ of an even-dimensional  nondegenerate subspace $U$ of $V$ with square discriminant in $\mathbb{F}_q$.
\end{itemize} 
As a consequence, it will be convenient to write $\mathcal{M} = \mathcal{M}_1 \cup \mathcal{M}_2$, where $\mathcal{M}_1$ comprises the subgroups in cases (a), (b) and (c). In turn, we will write $\Sigma(x) = \Sigma_1(x) + \Sigma_2(x)$ for all nontrivial $x \in G$, where $\Sigma(x)$ is defined as in \eqref{e:sigma} and $\Sigma_1(x)$ denotes the contribution from the subgroups in $\mathcal{M}_1$.

Using {\sc Magma}, we can handle the cases $(n,q) = (7,3)$, $(7,5)$ and $(9,3)$ directly (see Remark \ref{r:matrix} and Example \ref{e:magma}). Indeed, we compute $\Sigma(x) \leqs \a$, where $\a$ is defined as follows:
\[
\begin{array}{c|ccc}
(n,q) & (7,3) & (7,5) & (9,3) \\ \hline 
\a & 322/351 & 298/1125 & 1090/3321 
\end{array}
\]
So for the remainder, we will assume that
\begin{equation}\label{e:list1}
(n,q) \not\in \{(7,3), (7,5), (9,3)\}.
\end{equation}

\begin{lem}\label{l:fpr1}
If $H \in \mathcal{M}_1$ then ${\rm fpr}(x,G/H) < q^{(4-n)/2}$ for all nontrivial $x \in G$. 
\end{lem}

\begin{proof}
We may assume $x$ has prime order and thus the main theorem of \cite{Bur1} (which is proved in \cite{Bur2,Bur3,Bur4}) yields
\begin{equation}\label{e:fpr1}
{\rm fpr}(x,G/H) < |x^G|^{-\frac{1}{2}+\frac{1}{n}}
\end{equation}
since $(n,q) \ne (7,3)$. The conjugacy classes of elements of prime order in $G$ are described in \cite[Section 3.5]{BG_book} and it is straightforward to show that 
\[
|x^G| \geqs \frac{1}{2}q^{\frac{1}{2}(n-1)}(q^{\frac{1}{2}(n-1)}-1),
\]
with equality if $x = (-I_{n-1},I_1)$ is an involution with a minus-type $(-1)$-eigenspace on the natural module. The result now follows by combining this with the upper bound in \eqref{e:fpr1}.
\end{proof}

\begin{prop}\label{p:1}
We have $\Sigma_1(x) < 1/12$ for all nontrivial $x \in G$.
\end{prop}

\begin{proof}
Fix a nontrivial element $x \in G$ and first assume $H \in \mathcal{M}$ is a subfield subgroup of type ${\rm O}_{n}(q_0)$, where $q=q_0^e$ and $e$ is an odd prime. There are at most $\log\log q$ possibilities for $e$ and for each choice there is a unique $G$-class of subgroups by \cite[Proposition 4.5.8]{KL}. Since there are no subgroups in $\mathcal{M}$ of type (b) or (d), by applying Lemma \ref{l:fpr1} we deduce that 
\[
\Sigma_1(x) < q^{\frac{1}{2}(4-n)}\log\log q < \frac{1}{12}
\]
as required. 

For the remainder, we may assume $q=p \equiv \pm 3 \imod{8}$, where $q \geqs 11$ if $n=7$ (see \eqref{e:list1}). First let $H \in \mathcal{M}_1$ be a maximal subgroup of type ${\rm O}_{1}(q) \wr S_n$. By \cite[Proposition 4.2.15]{KL}, there is a unique $G$-class of such subgroups and Lemma \ref{l:fpr1} implies that ${\rm fpr}(x,G/H) < q^{(4-n)/2}$. If $n \geqs 9$ then $\Sigma_1(x) = {\rm fpr}(x,G/H)$ and the result follows. Finally suppose $n=7$, $q \geqs 11$ and $H = {\rm Sp}_6(2)$. Here there are two $G$-classes of these subgroups (see \cite[Table 8.40]{BHR}, for example) and thus $\Sigma_1(x) < 3q^{(4-n)/2}<1/12$.
\end{proof}

\begin{prop}\label{p:2}
We have $\Sigma_2(x) < 11/12$ for all nontrivial $x \in G$.
\end{prop}

\begin{proof}
Let $H = G_U$ be the stabilizer of a nondegenerate $m$-space $U$, where $m$ is even and the restriction of the defining quadratic form on $V$ has square discriminant. The latter condition uniquely determines the Witt index of $U$ and by applying \cite[Proposition 4.1.6]{KL} we deduce that there is a unique $G$-class of subgroups for each $m$, hence a total of at most $(n-1)/2$ subgroups in $\mathcal{M}_2$.

Since $H$ is a subspace subgroup, we cannot appeal to the upper bound in \eqref{e:fpr1}. In its place, we work with the bound in \cite[Proposition 3.16]{GK}, which gives
\[
{\rm fpr}(x,G/H) < (2q+1)q^{\frac{1}{2}(1-n)}+q^{\frac{1}{2}(m+1-n)}+q^{-m}
\] 
if $2 \leqs m \leqs n-5$ and
\[
{\rm fpr}(x,G/H) < (2q+1)q^{\frac{1}{2}(1-n)}+q^{\frac{1}{2}(2-m)}+q^{m-n}
\]
for $m \in \{n-3,n-1\}$.
Putting these bounds together, we deduce that
\[
\Sigma_2(x) < \left(\frac{1}{2}(n-1)(2q+1)+q(q+1)\right)q^{\frac{1}{2}(1-n)}+q^{-3} +a+b,
\]
where
\[
a = q^{-1}\sum_{i=0}^{\infty}q^{-i} < \frac{1}{q-1},\;\; b = q^{-2}\sum_{i=0}^{\infty}q^{-2i} < \frac{1}{q^2-1}.
\]
Since we may assume $(n,q) \ne (7,3), (9,3)$ (see \eqref{e:list1}), it is straightforward to check that $\Sigma_2(x)<11/12$ as required. 
\end{proof}

By combining Propositions \ref{p:1} and \ref{p:2}, we obtain the following result.

\begin{thm}\label{t:o_odd}
The conclusion to Theorem \ref{t:main2} holds if $G = \O_n(q)$, where $n \geqs 7$ is odd. 
\end{thm}

\subsection{Even dimensional orthogonal groups}\label{s:o_even}

In this section we assume $G = {\rm P\O}_n^{\e}(q)$, where $n \geqs 8$ is even and $q \geqs 3$ is odd. As before, let $P$ be a Sylow $2$-subgroup of $G$ and let $H \in \mathcal{M}$ be a maximal overgroup of $P$ in $G$. By Theorem \ref{t:kon} we have $n(H,P) = 1$ and using \cite{LS85,Mas} we deduce that one of the following holds:
\begin{itemize}\addtolength{\itemsep}{0.2\baselineskip}
\item[{\rm (a)}] $H$ is a subfield subgroup of type ${\rm O}_{n}^{\e}(q_0)$, where $q=q_0^e$ for some odd prime $e$.
\item[{\rm (b)}] $H$ is the stabilizer of an orthogonal decomposition $V = V_1 \perp \cdots \perp V_b$, where the $V_i$ are isometric nondegenerate $a$-spaces with $a = 2^k$, $k \geqs 0$.
\item[{\rm (c)}] $G = {\rm P\O}_{8}^{+}(q)$, $q=p \equiv \pm 3 \imod{8}$ and $H = \O_8^{+}(2)$.
\item[{\rm (d)}] $H = G_U$ is the stabilizer in $G$ of a nondegenerate subspace $U$ of $V$.
\end{itemize} 
Write $\mathcal{M} = \mathcal{M}_1 \cup \mathcal{M}_2$, where $\mathcal{M}_1$ comprises the subgroups of type (a), (b) and (c). Similarly, write $\Sigma(x) = \Sigma_1(x)+\Sigma_2(x)$ for all nontrivial $x \in G$.

Using {\sc Magma}, we compute $\Sigma(x) \leqs 986/3731$ when $(n,q) = (8,3)$ and $\Sigma(x) \leqs 17629/203763$ for $(n,q) = (8,5)$. Therefore, for the remainder we will assume
\begin{equation}\label{e:list0}
(n,q) \not\in \{(8,3),(8,5)\}.
\end{equation}

\begin{lem}\label{l:fpr2}
If $H \in \mathcal{M}_1$ then ${\rm fpr}(x,G/H) < q^{5-n}$ for all nontrivial $x \in G$. 
\end{lem}

\begin{proof}
Suppose $x \in G$ has prime order and note that \eqref{e:fpr1} holds via the main theorem of \cite{Bur1} (here we are using the fact that $(\e,n,q) \ne (+,8,3)$). By inspecting \cite[Section 3.5]{BG_book}, it is straightforward to show that
\[
|x^G| \geqs \frac{(q^{n-2}-1)(q^{n/2}+1)(q^{n/2-2}-1)}{q^2-1},
\]
with equality if and only if $\e=-$ and $x$ is a long root element (that is, $x$ is a unipotent element with Jordan form $(J_2^2,J_1^{n-4})$ on $V$). The result now follows by combining these two bounds. 
\end{proof}

\begin{prop}\label{p:21}
We have $\Sigma_1(x) < 1/22$ for all nontrivial $x \in G$.
\end{prop}

\begin{proof}
This is very similar to the proof of Proposition \ref{p:1}. First assume $H$ is a subfield subgroup of type ${\rm O}_{n}^{\e}(q_0)$, where $q=q_0^e$ and $e$ is an odd prime. There are at most $\log\log q$ possibilities for $e$, with a unique $G$-class of subgroups for each $e$ (see \cite[Proposition 4.5.8]{KL}). Next suppose $H$ is the stabilizer of an orthogonal decomposition as in (b) above. Now $G$ has two classes of subgroups of type ${\rm O}_1(q) \wr S_n$ (see \cite[Proposition 4.2.15]{KL}) and at most one class of subgroups of type ${\rm O}_{a}^{\e'}(q) \wr S_{n/a}$ for each divisor $a = 2^k$ of $n$ and each choice of sign $\e' = \pm$, where $k \geqs 1$. Since there are at most $\log n - 1$ possibilities for $a$, it follows that $\mathcal{M}_1$ contains  at most $2\log n$ subgroups of type (b). Finally, if $q=p \equiv \pm 3 \imod{8}$ then \cite[Table 8.50]{BHR} indicates that ${\rm P\O}_8^{+}(q)$ has $4$ classes of maximal subgroups isomorphic to $\O_{8}^{+}(2)$. 

In view of Lemma \ref{l:fpr2} and \eqref{e:list0}, we conclude that 
\[
\Sigma_1(x) < q^{5-n}(\delta\log\log q + 2\log n + 4) \leqs \frac{1}{22}
\]
for all nontrivial $x \in G$, where $\delta=1$ if $q \geqs 27$, otherwise $\delta=0$.
\end{proof}

We will need the following lemma in order to obtain an effective upper bound on $\Sigma_2(x)$.

\begin{lem}\label{l:fpr3}
Let $H$ be the stabilizer in $G$ of a $1$-dimensional nondegenerate subspace of the natural module. Then ${\rm fpr}(x,G/H) \leqs (q^2-1)^{-1}$ for all nontrivial $x \in G$. 
\end{lem}

\begin{proof}
Let $V$ be the natural module for $G$ and set $\O = G/H$, where $H$ is an almost simple group with socle $\O_{n-1}(q)$. Then  
\[
|\O| = \frac{1}{2}q^{\frac{1}{2}(n-2)}(q^{\frac{1}{2}n}-\e)
\] 
and we may identify $\O$ with the set of $1$-dimensional subspaces $\la v \ra$ such that $Q(v)$ is a square in $\mathbb{F}_q$, where $Q$ is the defining quadratic form on $V$. It suffices to verify the bound for elements of prime order, so let us assume $x \in G$ has prime order $r$.

First assume $r=2$. Here $|C_{\O}(x)|$ is maximal when $\e=+$ and $x$ is an involution of the form $(-I_2,I_{n-2})$, where both eigenspaces are nondegenerate minus-type spaces. By calculating the number of appropriate $1$-spaces in each eigenspace we deduce that
\[
|C_{\O}(x)| \leqs \frac{1}{2}(q+1)+\frac{1}{2}q^{\frac{1}{2}(n-4)}(q^{\frac{1}{2}(n-2)}+1)
\]
and the result follows. Next assume $r=p$, so $x$ is unipotent. If $x$ does not have Jordan form $(J_3,J_1^{n-3})$ on $V$, then one can check that the bounds on $|C_{\O}(x)|$ given in the proof of \cite[Lemma 5.30]{BG_FPR} are sufficient. Now assume $x = (J_3,J_1^{n-3})$. There are two conjugacy classes of this form in $G$, and there are also two classes of elements in $H$ with Jordan form $(J_3,J_1^{n-4})$ on the natural module for $H$. The two $H$-classes are not fused in $G$, so 
\[
|x^G \cap H| = |x^H| \leqs \frac{|{\rm SO}_{n-1}(q)|}{2q^{n-3}|{\rm SO}_{n-4}^{+}(q)|},\;\; |x^G| = \frac{|{\rm SO}_n^{\e}(q)|}{2q^{n-2}|{\rm SO}_{n-3}(q)|}
\]
and thus
\[
{\rm fpr}(x,G/H) \leqs \frac{q^{(n-4)/2}+1}{q^{n/2}-\e} \leqs \frac{1}{q^2-1}.
\]
Finally, if $r \ne p$ is odd then it is easy to check that the bounds presented in the proof of \cite[Lemma 5.30]{BG_FPR} are sufficient.
\end{proof}

\begin{prop}\label{p:22}
We have $\Sigma_2(x) < 21/22$ for all nontrivial $x \in G$.
\end{prop}

\begin{proof}
Let $H$ be the stabilizer of a nondegenerate $m$-space and let $x \in G$ be nontrivial. First assume $m = 2\ell+1$ is odd, where $1 \leqs \ell \leqs \lfloor (n-2)/4 \rfloor$. By \cite[Proposition 4.1.6]{KL}, there are two conjugacy classes of subgroups of this type and \cite[Proposition 3.16]{GK} gives
\[
{\rm fpr}(x,G/H) < (2q+1)q^{\frac{1}{2}(2-n)}+q^{\frac{1}{2}(2\ell+2-n)}+q^{-2\ell-1}.
\]
Now Lemma \ref{l:fpr3} implies that ${\rm fpr}(x,G/H) \leqs (q^2-1)^{-1}$ if $m=1$ or $n-1$, so the contribution to $\Sigma_2(x)$ from stabilizers of odd-dimensional spaces is less than
\[
\a_1 = \sum_{\ell=1}^{\lfloor (n-2)/4 \rfloor}2\left((2q+1)q^{\frac{1}{2}(2-n)}+q^{\frac{1}{2}(2\ell+2-n)}+q^{-2\ell-1}\right) + 2(q^2-1)^{-1}.
\]
Similarly, the contribution from the stabilizers of even-dimensional spaces is less than
\[
\a_2 = \sum_{\ell=1}^{\lfloor n/4 \rfloor}2\left((2q+1)q^{\frac{1}{2}(2-n)}+q^{\frac{1}{2}(2\ell+2-n)}+q^{-2\ell}\right)
\]
and we deduce that
\begin{align*}
\Sigma_2(x) < & \; (n-1)(2q+1)q^{\frac{1}{2}(2-n)}+2q^{-2}a + 2q^{\frac{1}{2}(2\lfloor n/4\rfloor+2-n)}a \\
& +2q^{\frac{1}{2}(2\lfloor (n-2)/4\rfloor+2-n)}a+2(q^2-1)^{-1}
\end{align*}
where 
\[
a = \sum_{i=0}^{\infty}q^{-i} = \frac{q}{q-1}.
\]
In view of \eqref{e:list0}, one can check that this upper bound is sufficient unless $(n,q) = (10,3)$ or $(12,3)$. 

Suppose $(n,q) = (10,3)$. If $\e=+$ then the main theorem of \cite{Mas} implies that each $\mathcal{M}_2$ is of type ${\rm O}_1(3) \perp {\rm O}_9(3)$ and thus Lemma \ref{l:fpr3} yields $\Sigma_2(x) \leqs 2/(3^2-1)=1/4$. Similarly, if $\e=-$ then $\Sigma_2(x) = {\rm fpr}(x,G/H)$ with $H$ of type ${\rm O}_{2}^{-}(3) \perp {\rm O}_{8}^{+}(3)$ and the bound in \cite[Proposition 3.16]{GK} is sufficient. Finally, suppose $(n,q) = (12,3)$. If $\e=+$ then $\mathcal{M}_2 = \{H\}$ with $H$ of type ${\rm O}_{4}^{+}(3) \perp {\rm O}_{8}^{+}(3)$ and the result quickly follows. On the other hand, if $\e=-$ then there are $6$ subgroups in $\mathcal{M}_2$: two each of type ${\rm O}_m(3) \perp {\rm O}_{12-m}(3)$ with $m \in \{1,3\}$, together with subgroups of type ${\rm O}_{m}^{-}(3) \perp {\rm O}_{12-m}^{+}(3)$ for $m \in \{2,4\}$. A routine calculation working with the bounds in Lemma \ref{l:fpr3} and \cite[Proposition 3.16]{GK} now establishes the result.
\end{proof}

\begin{thm}\label{t:o_even}
The conclusion to Theorem \ref{t:main2} holds if $G = {\rm P\O}_n^{\e}(q)$, where $n \geqs 8$ is even. 
\end{thm}

\subsection{Symplectic groups}\label{s:symp}

Let $G = {\rm PSp}_n(q)$, where $n \geqs 4$ and $q \geqs 3$ is odd. As usual, let $P$ be a Sylow $2$-subgroup of $G$ and let $H \in \mathcal{M}$ be a maximal overgroup of $P$ in $G$. By inspecting \cite{LS85,Mas} we see that one of the following holds:
\begin{itemize}\addtolength{\itemsep}{0.2\baselineskip}
\item[{\rm (a)}] $H$ is a subfield subgroup of type ${\rm Sp}_{n}(q_0)$, where $q=q_0^e$ for some odd prime $e$.
\item[{\rm (b)}] $H$ is the stabilizer of an orthogonal decomposition $V = V_1 \perp \cdots \perp V_b$, where the $V_i$ are nondegenerate $a$-spaces with $a = 2^k$, $k \geqs 1$.  
\item[{\rm (c)}] $H = G_U$ is the stabilizer in $G$ of a nondegenerate subspace $U$ of $V$ with $\dim U < n/2$.
\item[{\rm (d)}] $n=4$, $q=p \equiv \pm 3 \imod{8}$ and $H = 2^4.A_5$. 
\end{itemize} 

By Theorem \ref{t:kon} we have $n(H,P) \leqs |N_G(P):P| = 3^t$, where either $t=0$, or $q \equiv \pm 3 \imod{8}$ and $t$ is the number of nonzero digits in the binary expansion of $n$. In fact, we have the following result.

\begin{lem}\label{l:sp1}
Let $H \in \mathcal{M}$ be a maximal overgroup of $P$. Then $n(H,P)=1$.
\end{lem}

\begin{proof}
In view of Theorem \ref{t:kon}, we may assume $q \equiv \pm 3 \imod{8}$. 
In (a) we have $H = {\rm PSp}_n(q_0)$, so Theorem \ref{t:kon} implies that $|N_G(P):P| = |N_H(P):P|$ and thus $n(H,P) = 1$. Similarly, $|P|=2^6$ and $N_G(P) = P{:}3 < H$ in case (d), whence $n(H,P)=1$. 

Next consider (c). As a $P$-module we have $V = V_1 \oplus \cdots \oplus V_t$, where each $V_i$ is an irreducible submodule of dimension $2^{a_i}$ with $a_1>a_2 > \cdots >a_t \geqs 0$. It follows that $P$ fixes at most one subspace of any given dimension and thus $n(H,P)=1$. 

Finally, let us assume $n = ab$ and $H$ is the stabilizer of an orthogonal decomposition $V = V_1 \perp \cdots \perp V_b$, where each $V_i$ is a nondegenerate $a$-space and $a=2^k$ with $k \geqs 1$. Here it is convenient to work in the quasisimple group $G={\rm Sp}_n(q)$, in which case $H = {\rm Sp}_a(q) \wr S_{b}$ and it suffices to show that $N_G(P) \leqs H$. If we write $b = 2^{b_1}+\cdots+2^{b_t}$ with $b_1>b_2> \cdots > b_t \geqs 0$ then $P = P_1 \times \cdots \times P_t$, where $P_i$ is a Sylow $2$-subgroup of $H_i = {\rm Sp}_{a}(q) \wr S_{2^{b_i}}$ and $P_i$ acts transitively on the set of $2^{b_i}$ $a$-spaces in the corresponding orthogonal decomposition preserved by $H_i$. Therefore, we may assume $b$ is a $2$-power and $P$ acts transitively on $\{V_1, \ldots, V_b\}$. Set $P_0 = P \cap {\rm Sp}_a(q)^b$ and observe that there exists an element $z \in N_G(P_0) \cap {\rm Sp}_a(q)^b$ of order $3$, which is centralized by a Sylow $2$-subgroup of $S_b$. Therefore, $z \in N_G(P) \cap H$ and thus $N_G(P) = P{:}\la z \ra \leqs H$ as required.
\end{proof}

For now we will assume $n \geqs 6$, postponing the analysis of the case $n=4$ to the end of the section. Write $\mathcal{M} = \mathcal{M}_1 \cup \mathcal{M}_2$ and $\Sigma(x) = \Sigma_1(x) + \Sigma_2(x)$, where $\mathcal{M}_1$ comprises the subgroups of type (a) and (b) as above.

With the aid of {\sc Magma}, we compute $\Sigma(x) \leqs \a$ in the following cases:
\[
\begin{array}{c|ccccc}
(n,q) & (6,3) & (6,5) & (8,3) & (8,5) & (10,3) \\ \hline 
\a & 2/13 & 10/217 & 6/205 & 14/4069 & 826/7381
\end{array}
\]
Therefore, for the remainder we will assume
\begin{equation}\label{e:list2}
(n,q) \not\in \{(6,3),(6,5),(8,3),(8,5),(10,3)\}.
\end{equation}

\begin{lem}\label{l:fpr4}
Suppose $n \geqs 6$, $H \in \mathcal{M}_1$ and $x \in G$ is nontrivial and not a transvection. Then 
\[
{\rm fpr}(x,G/H) < q^{\ell-n},
\]
where $\ell =6$ if $H$ is of type ${\rm Sp}_{n/2}(q) \wr S_2$, otherwise $\ell=4$.
\end{lem}

\begin{proof}
First assume $H$ is not of type ${\rm Sp}_{n/2}(q) \wr S_2$. Here \eqref{e:fpr1} holds by the main theorem of \cite{Bur1} and the result follows since 
\[
|x^G| \geqs \frac{|{\rm Sp}_n(q)|}{|{\rm Sp}_{n-2}(q)||{\rm Sp}_{2}(q)|} = \frac{q^{n-2}(q^n-1)}{q^2-1},
\]
with equality if $x = (-I_2,I_{n-2})$. Similarly, if $H$ is of type ${\rm Sp}_{n/2}(q) \wr S_2$ then $n \geqs 8$, 
\[
{\rm fpr}(x,G/H) < |x^G|^{-\frac{1}{2}+\frac{2}{n}}
\]
by \cite{Bur1} and the result follows once again by applying the above lower bound on $|x^G|$.
\end{proof} 

\begin{prop}\label{p:31}
If $n \geqs 6$ then $\Sigma_1(x) < 1/43$ for all nontrivial $x \in G$.
\end{prop}

\begin{proof}
Let $x \in G$ be nontrivial and write $\Sigma_1(x) = \a(x)+\b(x)$, where $\a(x)$ is the contribution to $\Sigma_1(x)$ from subgroups of type (a). We will show that $\a(x)<1/200$ and $\b(x)<1/45$, which implies that $\Sigma_1(x) < 1/43$ as required.

First assume $H$ is a subfield subgroup of type ${\rm Sp}_n(q_0)$, where $q=q_0^e$ and $e \geqs 3$ is a prime. Here there are at most $\log\log q$ possibilities for $e$ and a unique $G$-class for each $e$ (see \cite[Proposition 4.5.4]{KL}). In addition, $n(H,P)=1$ by Lemma \ref{l:sp1}.   

Suppose $x$ is a transvection. There are two conjugacy classes of transvections in both $G$ and $H$, and the two $H$-classes are not fused in $G$ (see the proof of \cite[Proposition 5.5.8]{BG_book}, for example). Therefore,
\[
{\rm fpr}(x,G/H) = \frac{|x^G \cap H|}{|x^G|} = \frac{q_0^n-1}{q^n-1} < 2q^{-\frac{2}{3}n}
\]
and thus $\a(x) < 2q^{-2n/3}\log\log q < 1/200$. On the other hand, if $x$ is not a transvection then Lemma \ref{l:fpr4} yields $\a(x) < q^{4-n}\log\log q < 1/200$. 

For the remainder, let us assume $H \in \mathcal{M}_1$ is a subgroup of type ${\rm Sp}_{a}(q) \wr S_{n/a}$, where $a = 2^k$ and $k \geqs 1$. By \cite[Proposition 4.2.10]{KL}, there is a unique class of such subgroups for each $k$ and Lemma \ref{l:sp1} gives $n(H,P)=1$. Also note that there are at most $\log(n/2)$ possibilities for $a$. 

If $x$ is a transvection then $|x^G| = (q^n-1)/2$ and $H$ contains $n(q^a-1)/a$ transvections, whence
\[
{\rm fpr}(x,G/H) \leqs \frac{2n}{a}\left(\frac{q^a-1}{q^n-1}\right)
\]
and in view of \eqref{e:list2} we deduce that
\[
\b(x) < \left(\frac{4}{q^{n/2}+1}\right)\log(n/2)< \frac{1}{45}.
\]
Now assume $x$ is not a transvection. If $n \equiv 2 \imod{4}$ then $a=2$ is the only possibility and thus Lemma \ref{l:fpr4} implies that $\b(x) < q^{4-n} < 1/45$. On the other hand, if $n \equiv 0 \imod{4}$ then $\b(x) < q^{4-n}(\log(n/2)-1)+q^{6-n}$ and once again this bound is sufficient.
\end{proof}

\begin{prop}\label{p:33}
If $n \geqs 6$ then $\Sigma_2(x) < 42/43$ for all nontrivial $x \in G$.
\end{prop}

\begin{proof}
Let $H = G_U$ be the stabilizer of a nondegenerate $m$-space $U$, where $m<n/2$ is even. By \cite[Proposition 4.1.3]{KL}, there is a unique $G$-class of subgroups for each $m$. In addition, Lemma \ref{l:sp1} gives $n(H,P)=1$.

Set $m = 2\ell$. Then \cite[Proposition 3.16]{GK} gives
\[
{\rm fpr}(x,G/H) < 2q^{\frac{1}{2}(4-n)}+q^{-\frac{1}{2}n}+q^{-\ell}+q^{2\ell-n}
\]
and thus
\begin{align*}
\Sigma_2(x) & <  \sum_{\ell=1}^{\lfloor n/4 \rfloor} \left(2q^{\frac{1}{2}(4-n)}+q^{-\frac{1}{2}n}+q^{-\ell}+q^{2\ell-n}\right) \\
& < \frac{1}{4}n(2q^2+1)q^{-\frac{1}{2}n}+q^{-1}\sum_{i=0}^{\infty} q^{-i} + q^{2\lfloor n/4 \rfloor-n}\sum_{i=0}^{\infty} q^{-2i} \\
& < \frac{1}{4}n(2q^2+1)q^{-\frac{1}{2}n}+\frac{1}{q-1}+q^{2\lfloor n/4 \rfloor-n}\left(\frac{q^2}{q^2-1}\right).
\end{align*}
One can check that this bound is sufficient.
\end{proof}

By combining Propositions \ref{p:31} and \ref{p:33}, we obtain the following result.

\begin{thm}\label{t:symp6}
The conclusion to Theorem \ref{t:main2} holds if $G = {\rm PSp}_n(q)$ and $n \geqs 6$. 
\end{thm}

Finally, we handle the $4$-dimensional symplectic groups. 

\begin{thm}\label{t:symp4}
The conclusion to Theorem \ref{t:main2} holds if $G = {\rm PSp}_4(q)$.
\end{thm}

\begin{proof}
If $q=3$ then using {\sc Magma} \cite{Magma} we compute $\Sigma(x) \leqs 7/15$, so for the remainder we may assume $q \geqs 5$. By the main theorem of \cite{LS91} we have ${\rm fpr}(x,G/H) \leqs 4/3q$ for all $H \in \mathcal{M}$ and by applying Lemma \ref{l:sp1} and inspecting \cite[Table 8.12]{BHR} we observe that there are at most $2+\log\log q$ subgroups in $\mathcal{M}$. Therefore,
\[
\Sigma(x) \leqs (2+\log\log q) \cdot \frac{4}{3q} <1  
\]
and the result follows.
\end{proof}

\subsection{Linear and unitary groups}\label{s:lu}

In this final section we complete the proof of Theorem \ref{t:main} by handling the linear and unitary groups. The low dimensional groups $G = {\rm L}_n^{\e}(q)$ with $n \leqs 5$ require special attention and they are treated separately at the end of the section.

\begin{thm}\label{t:psun}
The conclusion to Theorem \ref{t:main2} holds if $G = {\rm U}_n(q)$ and $n \geqs 6$.
\end{thm}

\begin{proof}
Let $H \in \mathcal{M}$. By applying \cite{LS85,Mas} we deduce that one of the following holds:
\begin{itemize}\addtolength{\itemsep}{0.2\baselineskip}
\item[{\rm (a)}] $H$ is a subfield subgroup of type ${\rm GU}_n(q_0)$, where $q=q_0^e$ and $e \geqs 3$ is a prime.
\item[{\rm (b)}] $H$ is the stabilizer of an orthogonal decomposition $V = V_1 \perp \cdots \perp V_b$, where the $V_i$ are nondegenerate $a$-spaces with $a = 2^k$, $k \geqs 0$.  
\item[{\rm (c)}] $H$ is the stabilizer of a nondegenerate $m$-space with $m<n/2$.
\end{itemize}
Let $x \in G$ be nontrivial and write $\mathcal{M} = \mathcal{M}_1 \cup \mathcal{M}_2$ and $\Sigma(x) = \Sigma_1(x) + \Sigma_2(x)$, where $\mathcal{M}_1$ comprises the subgroups in (a) and (b). If $(n,q) = (6,3)$ then a {\sc Magma} computation yields $\Sigma(x) \leqs 5/351$, so for the remainder we may assume $(n,q) \ne (6,3)$. We claim that $\Sigma_i(x) <1/2$ for $i=1,2$ and thus $\Sigma(x)<1$ as required.

First assume $H \in \mathcal{M}_1$ and note that \eqref{e:fpr1} holds (see \cite{Bur1}). Then 
\[
|x^G| > \frac{(q^{n-1}-1)(q^n-1)}{q+1}
\]
and by applying \eqref{e:fpr1} we deduce that ${\rm fpr}(x,G/H)<q^{7/2-n}$. In addition, Theorem \ref{t:kon} yields 
\[
n(H,P) \leqs |G:N_G(P)| = \frac{((q+1)_{2'})^{t-1}}{(n,q+1)_{2'}} \leqs \left(\frac{1}{2}(q+1)\right)^{\log n - 1},
\] 
where $t$ is the number of nonzero digits in the binary expansion of $n$. Note that $t<\log n$.

In (a), there are at most $\log\log q$ choices for $e$ and there are $(n,(q+1)/(q_0+1))$ distinct $G$-classes for each $e$ (see \cite[Proposition 4.5.3]{KL}). Similarly, there are at most $\log n$ choices for $a$ in (b) and by \cite[Proposition 4.2.9]{KL} we see that there is a unique conjugacy class of subgroups for each given $a$. Since $(n,(q+1)/(q_0+1)) \leqs (n,q+1)_{2'}$, it follows that
\[
\Sigma_1(x) < (\delta\log\log q + \log n)\left(\frac{1}{2}(q+1)\right)^{\log n - 1}q^{7/2-n},
\]
where $\delta=1$ if $q \geqs 27$, otherwise $\delta=0$. One can check that this yields $\Sigma_1(x)<1/2$.

Now assume $H \in \mathcal{M}_2$ is the stabilizer of a nondegenerate $m$-space, where $m<n/2$. Note that $n(H,P) = 1$ since $P$ fixes at most one subspace of $V$ of any given dimension. We divide the analysis into two cases according to the parity of $n$.

First assume $n \geqs 6$ is even. Here \cite[Proposition 3.16]{GK} gives 
\[
{\rm fpr}(x,G/H)<(2q^3+1)q^{1-n}+q^{m+d-n}+q^{-2m},
\]
where $d=1$ if $m=1$, otherwise $d=0$. Therefore,
\[
\Sigma_2(x)< (n/2-1)(2q^3+1)q^{1-n}+q^{2-n}+q^{-2}a+q^{-n/2-1}b,
\]
where 
\[
a = \sum_{i=0}^{\infty}q^{-2i} = \frac{q^2}{q^2-1},\;\; b = \sum_{i=0}^{\infty}q^{-i} = \frac{q}{q-1},
\]
and it is straightforward to check that $\Sigma_2(x) < 1/2$ as required.

Finally, suppose $n \geqs 7$ is odd. Here  
\[
{\rm fpr}(x,G/H)<(2q^3+1)q^{-n}+q^{m+d-n}+q^{-2m}
\]
by \cite[Proposition 3.16]{GK}, where $d=1$ if $m$ is even, otherwise $d=0$. Therefore,
\[
\Sigma_2(x)< \frac{1}{2}(n-1)(2q^3+1)q^{-n}+q^{1-n}+(q^{-2}+2q^{\frac{1}{2}(1-n)})a,
\]
where $a$ is defined as above. This yields $\Sigma_2(x)<1/2$ and the result follows.
\end{proof}

\begin{thm}\label{t:psln2}
The conclusion to Theorem \ref{t:main2} holds if $G = {\rm L}_n(q)$ and $n \geqs 6$.
\end{thm}

\begin{proof}
We need to show that $\Sigma(x) < 1$ unless $n$ is odd and $q=3$. So for the remainder of the proof we will assume $n$ is even if $q=3$. 

Let $H \in \mathcal{M}$. By inspecting \cite{LS85,Mas} we see that one of the following holds:
\begin{itemize}\addtolength{\itemsep}{0.2\baselineskip}
\item[{\rm (a)}] $H$ is a subfield subgroup of type ${\rm GL}_n(q_0)$, where $q=q_0^e$ and $e \geqs 3$ is a prime.
\item[{\rm (b)}] $H$ is the stabilizer of a decomposition $V = V_1 \oplus \cdots \oplus V_b$, where the $V_i$ are $a$-spaces with $a = 2^k$, $k \geqs 0$.
\item[{\rm (c)}] $H$ is the stabilizer of an $m$-space with $1 \leqs m<n$.
\end{itemize}
As before, write $\mathcal{M} = \mathcal{M}_1 \cup \mathcal{M}_2$ and $\Sigma(x) = \Sigma_1(x) + \Sigma_2(x)$, where $\mathcal{M}_1$ comprises the subgroups in (a) and (b). If $(n,q) = (6,3)$ then a direct computation gives $\Sigma(x) \leqs 235/1001$, so for the remainder we may assume $(n,q) \ne (6,3)$. 

By repeating the argument in the proof of Theorem \ref{t:psun}, it is straightforward to show that 
\[
\Sigma_1(x) <(\delta\log\log q + \log n)\left(\frac{1}{2}(q-1)\right)^{\log n - 1}q^{3-n} < \frac{1}{3},
\]
where $\delta=1$ if $q \geqs 27$, otherwise $\delta=0$. Therefore, to complete the proof we just need to verify the bound $\Sigma_2(x) \leqs 2/3$. 

Write $n = 2^{a_1} + \cdots + 2^{a_t}$, where $a_1 > a_2 > \cdots > a_t \geqs 0$. Then $P$ fixes an $m$-space if and only if $m = b_12^{a_1} + \cdots + b_t2^{a_t}$ for some $b_i \in \{0,1\}$. Moreover, if $m$ is of this form then $P$ fixes a unique $m$-space and thus $n(H,P) = 1$ for all $H \in \mathcal{M}_2$. Suppose $n \geqs 6$ is even, so $m \ne 1, n-1$. If $m \leqs n/2$ then  \cite[Proposition 3.1]{GK} gives ${\rm fpr}(x,G/H)<2q^{-m}$ and thus
\[
\Sigma_2(x) < 4q^{-2}\sum_{i=1}^{\infty}q^{-i} = \frac{4}{q(q-1)} \leqs \frac{2}{3}.
\]
Similarly, if $n \geqs 7$ is odd then \cite[Proposition 3.1]{GK} yields
\[
\Sigma_2(x) < \frac{4}{q(q-1)} + 2\left(\frac{1}{q}+\frac{1}{q^{n-1}}\right)
\]
and we deduce that $\Sigma_2(x) < 2/3$ for all $q \geqs 5$. The result follows.
\end{proof}

To complete the proof of Theorem \ref{t:main} for the groups $G = {\rm L}_n^{\e}(q)$ with $n \geqs 6$ we may assume $\e=+$, $n$ is odd and $q=3$, in which case $G = {\rm SL}_n(3)$. The following lemma handles the case where $x \in G$ is semisimple and the proof relies on a theorem of Gow \cite{Gow}, which is proved using character theory.

\begin{lem}\label{l:psl1}
Let $G= {\rm L}_n(3)$, where $n \geqs 7$ is odd, and let $P$ be a Sylow $2$-subgroup of $G$. If $x \in G$ is non-central and semisimple, then $G = \la P, x^g \ra$ for some $g \in G$.
\end{lem}

\begin{proof}
Let $y \in G$ be an element of order $(3^n-1)/2$ and note that each maximal overgroup of $P$ in $G$ is a parabolic subgroup (see \cite{Mas}). Since $y$ acts irreducibly on the natural module, it follows that $G = \la P, y^g \ra$ for all $g \in G$. By \cite[Lemma 7.5]{GLOST}, $P$ contains a regular semisisimple element $z$ and by combining the main theorems of \cite{Gow,GT} we deduce that $x \in y^Gz^G$, say $x = y^az^b$. Then $y^{ab^{-1}} = x^{b^{-1}}z^{-1}$ and the result now follows since $G = \la P, x^{b^{-1}}z^{-1}\ra = \la P,x^{b^{-1}}\ra$.
\end{proof}

\begin{rem}
The method used in the proof of Lemma \ref{l:psl1} can be applied more generally. For example, suppose $G = {\rm SL}_n(q)$, where $q$ is odd and $n \geqs 3$. Let $P$ be a Sylow $2$-subgroup of $G$ and let $y \in G$ be an element of order $(q^n-1)/(q-1)$. By \cite{Ber}, every maximal subgroup of $G$ containing $y$ is a field extension subgroup of type ${\rm GL}_{n/k}(q^k)$ for some prime divisor $k$ of $n$. Since none of these subgroups contain $P$, it follows that $G = \la P, y^g \ra$ for all $g \in G$ and so we can repeat the argument in the proof of Lemma \ref{l:psl1} to conclude that if $x \in G$ is non-central and semisimple, then $G = \la P, x^g \ra$ for some $g \in G$. But this still leaves us to deal with unipotent elements, which explains why we work with fixed point ratios in the proof of Theorem \ref{t:psln2}.
\end{rem}

Next we turn to unipotent elements. In the following lemma we use elementary linear algebra to handle transvections.

\begin{lem}\label{l:psl_trans}
Let $G= {\rm L}_n(3)$, where $n \geqs 7$ is odd, and let $P$ be a Sylow $2$-subgroup of $G$. If $x \in G$ is a transvection, then $G = \la P, x^g \ra$ for some $g \in G$.
\end{lem}

\begin{proof}
First observe that any two transvections in $G$ are conjugate and recall that each maximal overgroup of $P$ in $G$ acts reducibly on the natural module $V$. Therefore, it suffices to construct a specific transvection $y$ such that $\la P, y \ra$ acts irreducibly on $V$.  

The space $V$ decomposes as a direct sum $V = V_1 \oplus \cdots \oplus V_t$ of irreducible submodules for $P$, where $\dim V_i = 2^{a_i}$ and $a_1> a_2 > \cdots > a_t \geqs 0$. Without loss of generality, we may assume that each $V_i$ has a basis contained in the standard basis $\{e_1, \ldots, e_n\}$ for $V$, where we view $e_i$ as a column vector with $1$ in the $i$-th position and $0$ in all other positions.

Let $N$ be an $n \times n$ matrix with columns $c_1, \ldots, c_n$ such that every entry in $c_1$ is equal to $1$ and $c_i = \a_i c_1$ for $i \geqs 2$, where the nonzero scalars $\a_i \in \mathbb{F}_3$ are chosen so that $1 + \alpha_2 + \cdots + \alpha_n = 0$. Then $N$ is a rank one nilpotent matrix and thus $y = I_n + N$ is a transvection. Moreover, every non-diagonal entry of $y$ is nonzero and thus $\la P,y \ra$ acts irreducibly on $V$. As noted above, this implies that $G = \la P,y\ra$ and the proof is complete.
\end{proof}

\begin{thm}\label{t:psln3}
The conclusion to Theorem \ref{t:main} holds if $G = {\rm L}_n(3)$ and $n \geqs 7$ is odd.
\end{thm}

\begin{proof}
Let $x \in G$ be nontrivial. In view of Lemmas \ref{l:psl1} and \ref{l:psl_trans}, we may assume $x$ is a unipotent element of order $3$ with $\dim C_V(x) \leqs n-2$. As before, let $\mathcal{M}$ be the set of maximal overgroups of $P$ in $G$ and recall that each $H \in \mathcal{M}$ is the stabilizer of an $m$-space. Moreover, there is a unique $G$-class of such subgroups for each $m$ and we have $n(H,P) = 1$ since $P$ fixes at most one subspace of any given dimension. As before, it suffices to show that $\Sigma(x)<1$.

By applying the upper bound in \cite[Proposition 3.1]{GK}, we see that the contribution to $\Sigma(x)$ from the stabilizers of $m$-spaces with $m \ne 1,n-1$ is less than
\[
4\sum_{i=2}^{(n-1)/2}3^{-i} < \frac{2}{3}.
\]
Finally, suppose $H \in \mathcal{M}$ is the stabilizer of a $1$-space (or an $(n-1)$-space) and set $\O = G/H$, so $|\O| = (3^n-1)/2$. Since $C_{\O}(x)$ coincides with the set of $1$-spaces in $C_V(x)$ it follows that $|C_{\O}(x)| \leqs (3^{n-2}-1)/2$ and ${\rm fpr}(x,G/H) \leqs (3^{n-2}-1)/(3^n-1)$. Therefore,
\[
\Sigma(x) < \frac{2}{3}+2\left(\frac{3^{n-2}-1}{3^n-1}\right) < 1
\]
and the result follows.
\end{proof}

To complete the proof of Theorem \ref{t:main} we may assume $G = {\rm L}_n^{\e}(q)$ and $n \leqs 5$. 

\begin{thm}\label{t:psl2}
The conclusion to Theorem \ref{t:main2} holds if $G = {\rm L}_2(q)$.
\end{thm}

\begin{proof}
With the aid of {\sc Magma}, one can check that $\Sigma(x) \leqs 51/55$ for all nontrivial $x \in G$ when $q<101$ (maximal when $q=11$), so we may assume $q \geqs 101$. By Theorem \ref{t:kon} we have $|N_G(P):P| = 3$ if $q \equiv \pm 3 \imod{8}$, otherwise $N_G(P)=P$. Let $H \in \mathcal{M}$ be a maximal subgroup of $G$ containing $P$. Then by inspecting \cite{Mas}, we deduce that one of the following holds:
\begin{itemize}\addtolength{\itemsep}{0.2\baselineskip}
\item[{\rm (a)}] $H$ is a subfield subgroup of type ${\rm GL}_2(q_0)$, where $q=q_0^2$.
\item[{\rm (b)}] $H$ is a subfield subgroup of type ${\rm GL}_2(q_0)$, where $q=q_0^e$ and $e \geqs 3$ is a prime.
\item[{\rm (c)}] $H = D_{q-\e}$ and $q \equiv \e \imod{4}$.
\item[{\rm (d)}] $H = A_4$ and $q=p \equiv \pm 3, \pm 13 \imod{40}$.
\item[{\rm (e)}] $H = S_4$ and $q=p \equiv \pm 7 \imod{16}$.
\item[{\rm (f)}] $H = A_5$ and $q=p \equiv \pm 11, \pm 19 \imod{40}$.
\end{itemize}
Write $\mathcal{M} = \mathcal{M}_1 \cup \mathcal{M}_2$, where $\mathcal{M}_1$ comprises the subfield subgroups in (a). In addition, write $\Sigma(x) = \Sigma_1(x) + \Sigma_2(x)$ for all nontrivial $x \in G$, where $\Sigma_1(x)$ is the contribution from the subgroups in $\mathcal{M}_1$.

Suppose $q=q_0^2$, so $q \equiv 1 \imod{8}$ and thus $N_G(P)=P$. There are two $G$-classes of subfield subgroups $H$ of type ${\rm GL}_2(q_0)$ and the main theorem of \cite{LS91} yields ${\rm fpr}(x,G/H) \leqs 2/(q+1)$. Since $n(H,P) = 1$, it follows that
\[
\Sigma_1(x) \leqs \frac{4}{q+1}.
\]
Similarly, if $H$ is one of the subgroups labelled (b)-(f), then ${\rm fpr}(x,G/H) \leqs 4/3q$ by \cite{LS91} and we note that $n(H,P) \leqs 3$. Since there are at most $\log\log q$ possibilities for $e$ in (b), we deduce that 
\[
\Sigma_2(x) \leqs 3(\log\log q+3) \cdot \frac{4}{3q}
\]
and it is straightforward to check that $\Sigma(x)<1$ as required.
\end{proof}

We will need the following lemma for the $3$-dimensional groups.

\begin{lem}\label{l:fpr6}
Let $G = {\rm L}_3^{\e}(q)$ and let $H$ be a subfield subgroup of type ${\rm GL}_3^{\e}(q_0)$, where $q = q_0^e$ and $e$ is an odd prime. Then ${\rm fpr}(x,G/H) < q^{-2}$ for all nontrivial $x \in G$. 
\end{lem}

\begin{proof}
Let $x \in G$ be an element of prime order $r$. First assume $r=p$, so $x$ has Jordan form $(J_2,J_1)$ or $(J_3)$ on the natural module. If $x = (J_2,J_1)$ then 
\[
|x^G \cap H| = \frac{|{\rm GL}_3^{\e}(q_0)|}{q_0^3|{\rm GL}_1^{\e}(q_0)|^2} = (q_0+\e)(q_0^3-\e),\;\; |x^G| = (q+\e)(q^3-\e)
\]
and the result follows. Similarly, if $x = (J_3)$ then
\[
|x^G \cap H| \leqs \frac{|{\rm GL}_3^{\e}(q_0)|}{q_0^2|{\rm GL}_1^{\e}(q_0)|} = q_0(q_0^2-1)(q_0^3-\e),\;\; |x^G| \geqs \frac{|{\rm GL}_3^{\e}(q)|}{3q^2|{\rm GL}_1^{\e}(q)|} = \frac{1}{3}q(q^2-1)(q^3-\e)
\]
and once again it is straightforward to check that these bounds are sufficient.

Now assume $r \ne p$. If $x$ is regular, then the trivial bound $|x^G \cap H| \leqs |{\rm PGL}_3^{\e}(q_0)|$ combined with the lower bound
\[
|x^G| \geqs \frac{|{\rm GU}_3(q)|}{3(q+1)^3} = \frac{1}{3}q^3(q-1)(q^2-q+1)
\]
is good enough. On the other hand, if $x$ is non-regular then $|x^G \cap H| = q_0^2(q_0^2+\e q_0+1)$, $|x^G| = q^2(q^2+\e q+1)$ and the result follows.\end{proof}

\begin{thm}\label{t:psln1}
The conclusion to Theorem \ref{t:main2} holds if $G = {\rm L}_n^{\e}(q)$ and $3 \leqs n \leqs 5$.
\end{thm}

\begin{proof}
First assume $n=3$. If $q <29$ then with the aid of {\sc Magma} it is straightforward to check that $\Sigma(x) \leqs 10/13$ for all nontrivial $x \in G$. Therefore, we may assume $q \geqs 29$. Let $H \in \mathcal{M}$ and note that 
\[
n(H,P) \leqs |N_G(P):P| = \frac{(q-\e)_{2'}}{(3,q-\e)} \leqs \frac{1}{2}(q+1)
\]
by Theorem \ref{t:kon}. By inspecting \cite{Mas} we see that one of the following holds:
\begin{itemize}\addtolength{\itemsep}{0.2\baselineskip}
\item[{\rm (a)}] $H$ is a subfield subgroup of type ${\rm GL}_3^{\e}(q_0)$, where $q=q_0^e$ and $e \geqs 3$ is a prime.
\item[{\rm (b)}] $H$ is of type ${\rm GL}_1^{\e}(q) \wr S_3$ and $q \equiv \e \imod{4}$.
\item[{\rm (c)}] $\e=+$ and $H$ is the stabilizer of a $1$-space or a $2$-space.
\item[{\rm (d)}] $\e=-$ and $H$ is the stabilizer of a nondegenerate $1$-space.
\end{itemize}
Write $\mathcal{M} = \mathcal{M}_1 \cup \mathcal{M}_2$ and $\Sigma(x) = \Sigma_1(x) + \Sigma_2(x)$, where $\mathcal{M}_1$ comprises the subfield subgroups in (a). 

Let $H \in \mathcal{M}_1$ be a subfield subgroup as in (a). By inspecting \cite[Tables 8.3, 8.5]{BHR} we note that for each $e$ there are $m$ distinct $G$-classes of subgroups of this form, where
\[
m = \left(3,\frac{q-\e}{q_0-\e}\right) = \left\{\begin{array}{ll}
3 & \mbox{if $e=3$ and $q_0 \equiv \e \imod{3}$} \\
1 & \mbox{otherwise.}
\end{array}\right.
\]
If $m=3$ then $(3,q-\e)=3$ and thus $n(H,P) \leqs (q+1)/6$. Therefore, since there are at most $\log\log q$ possibilities for $e$, by applying Lemma \ref{l:fpr6} we deduce that 
\[
\Sigma_1(x) < \frac{1}{2}(q+1)q^{-2}\log\log q.
\]

Now let us consider the subgroups $H$ in $\mathcal{M}_2$. There is a unique conjugacy class of subgroups of type (b) and (d), and there are two classes in (c) (one for each dimension). Moreover, $P$ fixes a unique $1$-space and a unique $2$-space (in (d), the relevant $2$-space is simply the orthogonal complement of the $1$-space fixed by $P$), so we have $n(H,P) = 1$ in (c) and (d). Therefore, by applying the main theorem of \cite{LS91} we deduce that
\[
\Sigma_2(x) < (2 + (q+1)/2)\cdot \frac{4}{3q}
\]
and it is now straightforward to check that $\Sigma(x) < 1$ for all nontrivial $x \in G$.

Next assume $n=4$. If $q<11$ then we compute $\Sigma(x) \leqs 34/117$ for all nontrivial $x \in G$, so we may assume $q \geqs 11$. Note that Theorem \ref{t:kon} yields $N_G(P) = P$ and thus $n(H,P) = 1$ for all $H \in \mathcal{M}$. Then in the usual way, using \cite{LS91,Mas}, we deduce that
\[
\Sigma(x) < (\delta\log\log q + 6) \cdot \frac{4}{3q} < 1
\]
for all nontrivial $x \in G$, where $\delta=1$ if $q \geqs 27$, otherwise $\delta=0$. 

Finally, let us assume $n=5$. If $q<9$ then we compute $\Sigma(x) \leqs 82/121$ for all nontrivial $x \in G$, so we may assume $q \geqs 9$. Let $H \in \mathcal{M}$ and note that 
\[
n(H,P) \leqs |G:N_G(P)| = \frac{(q-\e)_{2'}}{(5,q-\e)} \leqs \frac{1}{2}(q+1)
\]
by Theorem \ref{t:kon}. Then by applying the main theorem of \cite{Mas}, we observe that one of the following holds:
\begin{itemize}\addtolength{\itemsep}{0.2\baselineskip}
\item[{\rm (a)}] $H$ is a subfield subgroup of type ${\rm GL}_5^{\e}(q_0)$, where $q=q_0^e$ and $e \geqs 3$ is a prime.
\item[{\rm (b)}] $H$ is of type ${\rm GL}_1^{\e}(q) \wr S_5$ and $q \equiv \e \imod{4}$.
\item[{\rm (c)}] $\e=+$ and $H$ is the stabilizer of a $1$-space or a $4$-space.
\item[{\rm (d)}] $\e=-$ and $H$ is the stabilizer of a nondegenerate $1$-space.
\end{itemize}
Write $\mathcal{M} = \mathcal{M}_1 \cup \mathcal{M}_2$ and $\Sigma(x) = \Sigma_1(x) + \Sigma_2(x)$, where $\mathcal{M}_1$ comprises the subgroups in (a) and (b). 

Let $H \in \mathcal{M}_1$. By combining the bound in \eqref{e:fpr1} (see \cite{Bur1}) with the lower bound 
\[
|x^G| \geqs \frac{|{\rm GL}_5^{\e}(q)|}{q^7|{\rm GL}_{3}^{\e}(q)||{\rm GL}_1^{\e}(q)|} = \frac{(q^4-1)(q^5-\e)}{q-\e}
\]
(equality if $x$ is unipotent with Jordan form $(J_2,J_1^3)$) we deduce that ${\rm fpr}(x,G/H) < q^{-2}$ for all nontrivial $x \in G$. Now $G$ has a unique class of subgroups as in (b). In addition, if $q=q_0^e$ then $G$ has $m$ classes of subfield subgroups of type ${\rm GL}_{5}^{\e}(q_0)$, where
\[
m = \left(5,\frac{q-\e}{q_0-\e}\right) = \left\{\begin{array}{ll}
5 & \mbox{if $e=5$ and $q_0 \equiv \e \imod{5}$} \\
1 & \mbox{otherwise.}
\end{array}\right.
\]
If $m=5$ then $(5,q-\e)=5$ and thus $n(H,P) \leqs (q+1)/10$. Since there are at most $\log\log q$ possibilities for $e$ we deduce that 
\[
\Sigma_1(x) < \frac{1}{2}(q+1)q^{-2}(\delta\log\log q + 1),
\]
where $\delta=1$ if $q \geqs 27$, otherwise $\delta=0$. 

Finally, suppose $H \in \mathcal{M}_2$. As before, $P$ fixes a unique $1$-space and a unique $4$-space, so $n(H,P) = 1$ and by applying the main theorem of \cite{LS91} we deduce that
\[
\Sigma_2(x) < 2\cdot \frac{4}{3q}.
\]
Since $q \geqs 9$, one can now check that $\Sigma(x) < 1$ for all nontrivial $x \in G$.
\end{proof}

By combining Theorems \ref{t:ex}, \ref{t:o_odd}, \ref{t:o_even}, \ref{t:symp6}, \ref{t:symp4}, \ref{t:psun}, \ref{t:psln2}, \ref{t:psl2} and \ref{t:psln1}, we conclude that the proof of Theorem \ref{t:main2} is complete. This reduces the proof of Theorem \ref{t:main} to the linear groups ${\rm L}_n(3)$ with $n \geqs 7$ odd and we handled this special case in Theorem \ref{t:psln3}. Therefore, the proof of Theorem \ref{t:main} is complete.

\end{document}